\documentclass[12pt,nosumlimits,nonamelimits]{amsart}
\usepackage{hyperref,mathtools}

\newtheorem*{thm*}{Theorem}
\newtheorem*{cor*}{Corollary}
\newtheorem{lem}{Lemma}

\newcommand{\id}{\mathrm{id}}

\DeclareMathOperator{\lh}{span}
\DeclareMathOperator{\clh}{\overline{span}}

\title{Total sequences and the lower frame inequality}
\author{Narutaka Ozawa}
\address{RIMS, Kyoto University, \mbox{606-8502} Japan}
\email{narutaka@kurims.kyoto-u.ac.jp}
\thanks{The author was partially supported by JSPS KAKENHI Grant Numbers 24K00527, 25H00588, 25H00593}

\date{\today}

\begin{document}
\begin{abstract}
We prove that every total sequence in a Hilbert space 
satisfies the lower frame inequality after scaling. 
This solves A. Kulikov's ``exercise'' at mathoverflow. 
\end{abstract}
\maketitle
Let $H$ be a Hilbert space. 
We will view a vector $v$ in $H$ 
as a ``column vector'', i.e., an operator 
from $\mathbb{C}$ into $H$. 
Thus $v^*w$ is identified with 
the inner product of $v$ and $w$, 
and $vw^*$ is a rank $\le1$ operator on $H$ 
with $\|vw^*\|=\|v\|\|w\|$. 
The symbol $\lh$ stands for the linear span and $\clh$ 
for the closure of it. 
A sequence $(v_n)_{n=1}^\infty$ in 
$H$ is said to be \emph{total} if 
$\clh\{ v_n : n\geq1\}=H$. 

The following theorem, in particular, solves A. Kulikov's ``exercise'' 
at mathoverflow 
(\href{https://mathoverflow.net/questions/496725/}{496725}). 
The author is grateful to A. Kulikov for sharing this problem 
and pointing out an error in the author's original 
solution at mathoverflow. 

\begin{thm*}
Let $( v_n )_{n=1}^\infty$ be a total sequence in a Hilbert space $H$. 
Then, there are sequences $(w_n)_{n=1}^\infty$ and $(z_n)_{n=1}^\infty$ 
in $H$ such that $z_n\in\lh\{ v_k : k\geq n/2\}$ and 
$$\id_H=\sum_{n=1}^\infty w_nz_n^*$$
in the strong operator topology. 
In particular, 
there is a sequence $(\lambda_n)_{n=1}^\infty$ 
of positive numbers such that $(\lambda_n^{1/2}v_n)_{n=1}^\infty$
satisfies the lower frame inequality: 
\[
\| x\|^2 \le \sum_{n=1}^\infty\lambda_n |v_n^*x|^2 
\]
for every $x\in H$, with the right hand side being possibly $\infty$.
\end{thm*}

\begin{proof}
Let $H_n\coloneqq\clh\{ v_k : k\geq n \}$ and $H_\infty\coloneqq\bigcap_n H_n$. 
Then $H=H_1\supset H_2\supset\cdots$ and 
$\dim(H_n\cap H_{n+1}^\perp)\le1$. 
For each $n$, if $H_n=H_{n+1}$, then set $x_n\coloneqq0$; 
else take a unit vector 
$x_n\in H_n\cap H_{n+1}^\perp$. 
Thus $\{ x_n : x_n\neq0\}$ is an orthonormal 
basis for $H_\infty^\perp$.
We also fix an orthonormal basis $\{ y_n\}$ for $H_\infty$, 
which is possibly empty, and 
set $u_{2n-1} \coloneqq x_n$ and $u_{2n}\coloneqq y_n$ 
or 0 if $y_n$ is not defined. 
Then $\{ u_n : u_n\neq0\}$ is an orthonormal basis 
for $H$ such that $u_n\in H_{\lceil n/2\rceil}$. 
We choose $z_n\in\lh\{ v_k : k\geq n/2\}$ such that 
$\| u_n - z_n \| < 3^{-n}$. 
Then $T\coloneqq\sum_n u_n (u_n-z_n)^*$ converges absolutely, 
$\|T\|<1$, and 
$$\id_H-T=\bigl(\sum_n u_nu_n^*\bigr)-T = \sum_n u_nz_n^*$$ 
in the strong operator topology. 
Thus $w_n\coloneqq(\id_H-T)^{-1}u_n$ meet our demand. 
This proves the first half of Theorem.

For each $n$, 
write $z_n = \sum_{k\geq n/2} \gamma^{(n)}_k v_k$ 
as a finite sum. 
Note that by the Cauchy--Schwarz inequality, 
$(\sum_{n=1}^\infty \beta_n)^2\le \sum_{n=1}^\infty 2^n\beta_n^2$ 
holds for every sequence $(\beta_n)_{n=1}^\infty$ 
of non-negative numbers, where 
both sides can be $\infty$. 
Thus for every $x\in H$ 
\[
\|x\|^2 = \|\sum_n w_n z_n^*x\|^2 
\le\sum_n 2^n\|w_n\|^2|z_n^*x|^2,
\]
while 
\[
|z_n^*x|^2=|\sum_{k\geq n/2}\gamma^{(n)}_k v_k^*x |^2
\le 2^{-(n/2)+1}\sum_{k\geq n/2} 2^k|\gamma^{(n)}_k|^2 |v_k^*x |^2;
\]
which amounts to 
\[
\|x\|^2 \le \sum_n\sum_{k\geq n/2}2^{(n/2)+k+1}\|w_n\|^2|\gamma^{(n)}_k|^2 |v_k^*x |^2
=\sum_{k=1}^\infty\lambda_k |v_k^*x |^2,
\]
where $\lambda_k\coloneqq\sum_{n=1}^{2k}2^{(n/2)+k+1}\|w_n\|^2|\gamma^{(n)}_k|^2$.
\end{proof}
\end{document}